\documentclass[11pt]{article}
\usepackage{amsmath}
\usepackage{amsthm}
\usepackage{amssymb}
\usepackage{mathrsfs}
\usepackage{bm}
\usepackage{pifont}
\usepackage{xcolor}

\usepackage{geometry}
\usepackage[colorlinks=true,
linkcolor=blue,citecolor=blue,
urlcolor=blue]{hyperref}

\makeatletter
\def\@seccntDot{.}
\def\@seccntformat#1{\csname the#1\endcsname\@seccntDot\hskip 0.5em}
\renewcommand\section{\@startsection{section}{1}{\z@}%
{18\p@ \@plus 6\p@ \@minus 3\p@}%
{9\p@ \@plus 6\p@ \@minus 3\p@}%
{\large\bfseries\boldmath}}
\renewcommand\subsection{\@startsection{subsection}{2}{\z@}%
{12\p@ \@plus 6\p@ \@minus 3\p@}%
{3\p@ \@plus 6\p@ \@minus 3\p@}%
{\bfseries\boldmath}}
\renewcommand\subsubsection{\@startsection{subsubsection}{3}{\z@}%
{12\p@ \@plus 6\p@ \@minus 3\p@}%
{\p@}%
{\bfseries\boldmath}}
\makeatother

\usepackage{microtype}

\theoremstyle{plain}
\newtheorem{theorem}{Theorem}[section]
\newtheorem{lemma}{Lemma}[section]

\newtheorem{conjecture}{Conjecture}[section]

\theoremstyle{definition}

\newtheorem{claim}{Claim}

\numberwithin{equation}{section}
\allowdisplaybreaks
\newcommand{\x}{{\bf x}}

\title{A proof of a conjecture on the distance spectral radius and maximum transmission of graphs}

\author{
Lele Liu\thanks{College of Science, University of Shanghai for Science and Technology, Shanghai 200093, China
(\texttt{ahhylau@outlook.com})}~,~~
Haiying Shan\thanks{\footnotesize School of Mathematical Sciences, Tongji University, Shanghai 200092, China
(\texttt{shan\_haiying@tongji.edu.cn})}~,~~
Changxiang He\thanks{College of Science, University of Shanghai for Science and Technology, Shanghai 200093, China
(\texttt{changxiang-he@163.com})}
}

\date{}

\begin{document}
\maketitle

\begin{abstract}
Let $G$ be a simple connected graph, and $D(G)$ be the distance matrix of $G$. Suppose that $D_{\max}(G)$ 
and $\lambda_1(G)$ are the maximum row sum and the spectral radius of $D(G)$, respectively. In this paper, 
we give a lower bound for $D_{\max}(G)-\lambda_1(G)$, and characterize the extremal graphs attaining the 
bound. As a corollary, we solve a conjecture posed by Liu, Shu and Xue.
\par\vspace{2mm}

\par\vspace{1mm}

\end{abstract}

\section{Introduction}

In this paper, we consider only simple, undirected graphs, i.e, undirected graphs without multiple edges or loops. 
A graph $G$ is a pair $(V(G), E(G))$ consisting of a set $V(G)$ of vertices and a set $E(G)$ of edges. The 
{\em distance matrix} $D(G)=(d_{ij})$ of a connected graph $G$ is the $|V(G)|\times |V(G)|$ matrix indexed by 
the vertices of $G$, where $d_{ij}=d(v_i,v_j)$ and $d(v_i,v_j)$ denotes the distance between the vertices $v_i$ and 
$v_j$, i.e., the length of a shortest path between $v_i$ and $v_j$. The largest eigenvalue of $D(G)$, denoted 
by $\lambda_1(G)$, is called the {\em distance spectral radius} of $G$. For graph notation and terminology 
undefined here we refer the reader to \cite{BondyMurty2008}.

The {\em transmission} of a vertex $v$ in $G$, denoted by $D_v(G)$ or simply by $D_v$, is the sum of the distances 
from $v$ to all other vertices in $G$. We denote the maximum and minimum transmission of $G$ by $D_{\max}$ and $D_{\min}$, 
respectively. That is,
\[
D_{\max}:=\max\{D_v\,|\,v\in V(G)\},~ 
D_{\min}:=\min\{D_v\,|\,v\in V(G)\}.
\]
If there is no risk of confusion we shall denote $D_{\max}(G)$ and $D_{\min}(G)$ simply by $D_{\max}$ and $D_{\min}$, 
respectively. A connected graph $G$ is {\em transmission-regular} if $D_{\max}=D_{\min}$; otherwise, $G$ is 
{\em non-transmission-regular}. The {\em Wiener index} $W(G)$ of a graph $G$ is the sum of the distances between 
all unordered pairs of vertices of $G$. Alternatively, $W(G)$ is half the sum of all the entries of the distance 
matrix of $G$, that is,
\[
W(G)=\frac{1}{2}\sum_{v\in V(G)} D_v.
\]

In \cite{AtikPanigrahi2018}, Atik and Panigrahi introduced a class of graphs called {\em DVDR graphs}. A connected 
graph $G$ on $n$ vertices is said to be {\em distinguished vertex deleted regular graph} (DVDR) if there exist a 
vertex $v$ in $G$ such that the degree $d(v)=n-1$ and $G-v$ is an regular graph. The vertex $v$ is said to be a 
{\em distinguished vertex} of the DVDR graph $G$. If $G-v$ is $r$-regular, we say $G$ is an $r$-{\em DVDR} graph.  

In \cite{LiuShuXue2018}, Liu, Shu and Xue studied the bounds of $D_{\max}(G)-\lambda_1(G)$ for connected 
non-transmission-regular graphs, and posed the following conjecture.

\begin{conjecture}[\cite{LiuShuXue2018}]\label{conj:D-lambda}
Let $G$ be a connected non-transmission-regular graph with $n$ vertices. Then
\[
D_{\max}(G)-\lambda_1(G)>\frac{1}{n+1}.
\]
\end{conjecture}

It is shown that the conjecture holds for trees in \cite{LiuShuXue2018}. In this paper, we confirm \autoref{conj:D-lambda} 
by proving the following strengthened results.

\begin{theorem}\label{thm:Main}
Let $G$ be a connected non-transmission-regular graph of order $n$. 
\begin{enumerate}
\item[$(1)$] If $n$ is odd, then
\[
D_{\max}(G)-\lambda_1(G)\geq\frac{n+1-\sqrt{(n-1)(n+3)}}{2}.
\] 
Equality holds if and only if $G\cong K_{1,2,\ldots,2}$.

\item[$(2)$] If $n$ is even, then
\[
D_{\max}(G)-\lambda_1(G)\geq\frac{n+2-\sqrt{n^2+4n-4}}{2}.
\] 
Equality holds if and only if $G$ is isomorphic to some $(n-4)$-DVDR graph.
\end{enumerate}
\end{theorem}

\section{Proof of \autoref{thm:Main}}

For a connected graph $G$ of order $n$, write $\sigma(G)$ for the difference between the maximum transmission and 
the distance spectral radius of $G$, i.e., 
\[
\sigma(G):=D_{\max}(G)-\lambda_1(G).
\]
In this section, we assume $G^*$ is a graph attaining the minimum of $\sigma(G)$ among connected non-transmission-regular 
graphs $G$ of order $n$. Let $\x$ be the positive unit eigenvector of $D(G^*)$ corresponding to $\lambda_1(G^*)$. For 
convenience, we denote $x_{\max}:=\max\{x_v\,|\, v\in V(G^*)\}$, $x_{\min}:=\min\{x_v\,|\, v\in V(G^*)\}$. Suppose that 
$D_{\max}$ and $W$ are the maximum transmission and Wiener index of $G^*$, respectively. An easy calculation implies that
\begin{align}\label{eq:D-lambda}
D_{\max}-\lambda_1(G^*) & = D_{\max}-2\sum_{\{u,v\}\subseteq V(G)} d(u,v) x_ux_v \nonumber \\
& = \sum_{u\in V(G)} (D_{\max}-D_u) x_u^2+\sum_{\{u,v\}\subseteq V(G)} d(u,v) (x_u-x_v)^2 \\
& \geq (nD_{\max}-2W) x_{\min}^2+\sum_{\{u,v\}\subseteq V(G)} d(u,v) (x_u-x_v)^2. \nonumber
\end{align}  

Before continuing the proof of \autoref{thm:Main}, we need the following two results.

\begin{lemma}
Let $K_{1,2,\ldots,2}$ be the complete $r$-partite graph on $n$ vertices with $n=2r-1$. Then
\[
\sigma(K_{1,2,\ldots,2})=\frac{n+1-\sqrt{(n-1)(n+3)}}{2}.
\]
\end{lemma}

\begin{proof}
Let $v$ be the distinguished vertex of $K_{1,2,\ldots,2}$. It is clear that the quotient matrix of 
$D(K_{1,2,\ldots,2})$ with respect to the equitable partition $\{v\}\cup (V(K_{1,2,\ldots,2})\backslash\{v\})$ is 
\[
Q=
\begin{bmatrix}
0 & n-1 \\
1 & n-1
\end{bmatrix}.
\]
As a consequence, $\lambda_1(K_{1,2,\ldots,2})$ is equal to the largest eigenvalue of $Q$, i.e.,
\[
\lambda_1(K_{1,2,\ldots,2})=\frac{n-1+\sqrt{(n-1)(n+3)}}{2}.
\]
Clearly, $D_{\max}(K_{1,2,\ldots,2})=n$, we immediately obtain
\[
\sigma(K_{1,2,\ldots,2})=D_{\max}(K_{1,2,\ldots,2})-\lambda_1(K_{1,2,\ldots,2})=\frac{n+1-\sqrt{(n-1)(n+3)}}{2},
\]
completing the proof.
\end{proof}

\begin{lemma}
Let $G$ be an $(n-4)$-DVDR graph on $n$ vertices. Then
\[
\sigma(G)=\frac{n+2-\sqrt{n^2+4n-4}}{2}.
\]
\end{lemma}

\begin{proof}
Let $v$ be the distinguished vertex of $G$. Then the quotient matrix of $D(G)$ with respect to the equitable 
partition $\{v\}\cup (V(G)\backslash\{v\})$ is 
\[
\begin{bmatrix}
0 & n-1 \\
1 & n
\end{bmatrix}.
\]
Therefore, we obtain that
\[
\lambda_1(G)=\frac{n+\sqrt{n^2+4n-4}}{2}.
\]
Clearly, $D_{\max}(G)=n+1$, we have
\[
\sigma(G)=D_{\max}(G)-\lambda_1(G)=\frac{n+2-\sqrt{n^2+4n-4}}{2},
\]
completing the proof.
\end{proof}

In the following, we denote 
\[
\sigma_n:=\frac{n+\gamma_n+\sqrt{(n+\gamma_n)^2-4\gamma_n}}{2},
\]
where $\gamma_n=1$ if $n$ is odd; $\gamma_n=2$ if $n$ is even. By simple algebra, we see
\begin{equation}\label{eq:equality-sigma-gamma}
\sigma_n^2-(n+\gamma_n)\sigma_n+\gamma_n=0.
\end{equation}
It is obvious that if $n$ is odd, then $\sigma_n=\sigma(K_{1,2,\ldots,2})$; if $n$ is even, then 
$\sigma_n=\sigma(G)$ for any $(n-4)$-DVDR graph $G$. Noting that $K_{1,2,\ldots,2}$ and each 
$(n-4)$-DVDR graph are both connected non-transmission-regular graphs, we see 
$\sigma(G^*)\leq\sigma_n$. \vspace{3mm}

\noindent {\bfseries Proof of \autoref{thm:Main}.}
Let $u$ and $v$ be two vertices such that $x_u=x_{\max}$ and $x_v=x_{\min}$, respectively. For short, 
write $\lambda_1:=\lambda_1(G^*)$. Using equation $D(G^*)\x=\lambda_1\x$ with respect to vertex $u$ we see
\begin{align*}
\lambda_1 x_u 
& = \sum_{w\in V(G^*)} d(u,w) x_w \\
& = \sum_{w\in V(G^*)\backslash\{v\}} d(u,w) x_w+d(u,v) x_v \\
& \leq (D_{\max}-d(u,v)) x_u + d(u,v)x_v,
\end{align*}
which yields that
\[
D_{\max}-\lambda_1\geq d(u,v)\left(1-\frac{x_v}{x_u}\right)\geq 1-\frac{x_v}{x_u}.
\]
On the other hand, $\sigma(G^*)=D_{\max}-\lambda_1\leq\sigma_n$. Hence, we obtain 
\begin{equation}\label{eq:up-ratio}
\frac{x_{\max}}{x_{\min}}=\frac{x_u}{x_v}\leq \frac{1}{1-\sigma_n}.
\end{equation}

Our proof hinges on the following five claims.

\begin{claim}\label{claim:nDmax-2W}
$nD_{\max}-2W=\gamma_n$.
\end{claim}

\noindent {\it Proof of \autoref{claim:nDmax-2W}.}
We prove this claim by contradiction. Let $n$ be odd, from \eqref{eq:D-lambda} and \eqref{eq:up-ratio} we have
\[
D_{\max}-\lambda_1 \geq 2x_{\min}^2\geq 2(1-\sigma_n)^2x_{\max}^2>\frac{2(1-\sigma_n)^2}{n}>\sigma_n,
\]
the last inequality follows from $\sigma_n^2-(n+1)\sigma_n+1=0$. This is a contradiction with 
$D_{\max}-\lambda_1\leq\sigma_n$.

Let $n$ be even. Noting that $2\,|\,(nD_{\max}-2W)$, we have $nD_{\max}-2W\geq 4$. It follows from 
\eqref{eq:D-lambda} and \eqref{eq:up-ratio} that
\[
D_{\max}-\lambda_1 \geq 4x_{\min}^2\geq 4(1-\sigma_n)^2x_{\max}^2>\frac{4(1-\sigma_n)^2}{n}>\sigma_n,
\]
a contradiction.

\begin{claim}\label{claim:Du-Dmin-Dv-Dmax}
$D_u>D_{\min}$ and $D_v<D_{\max}$.
\end{claim}

\noindent {\it Proof of \autoref{claim:Du-Dmin-Dv-Dmax}.}
By equation $D(G^*)\x=\lambda_1\x$ with respect to vertex $u$ we have
\[
\lambda_1 x_u=\sum_{w\in V(G^*)} d(w,u) x_w\leq D_u x_u,
\]
which implies $\lambda_1\leq D_u$. Since $\lambda_1>D_{\min}$, we obtain $D_u>D_{\min}$.
Similarly, we can prove that $D_v<D_{\max}$, completing the proof of this claim.

\begin{claim}\label{claim:ratio}
$\displaystyle\frac{x_{\max}}{x_{\min}}=\frac{1}{1-\sigma_n}$, $D_{\max}-\lambda_1=\sigma_n$, and there 
are $(n-1)$ vertices attaining $x_{\max}$.
\end{claim}

\noindent {\it Proof of \autoref{claim:ratio}.}
Since $D(G^*)\x=\lambda_1\x$, we have $\lambda_1\bm{1}\x=\bm{1}D(G^*)\x$, where $\bm{1}$ is the all-ones 
vector of dimension $n$. Hence,
\[
\lambda_1\sum_{w\in V(G^*)} x_w=\sum_{w\in V(G^*)} D_w x_w.
\]
It follows from \autoref{claim:nDmax-2W} that
\begin{align*}
(D_{\max}-\lambda_1)\sum_{w\in V(G^*)} x_w & =\sum_{w\in V(G^*)} (D_{\max}-D_w) x_w \\ 
& \geq (nD_{\max}-2W) x_{\min} \\
& =\gamma_n x_{\min}.
\end{align*}
Therefore we have
\begin{equation}\label{eq:D-lambda-lower-bound}
D_{\max}-\lambda_1\geq\frac{\gamma_n x_{\min}}{\sum_{w\in V(G^*)} x_w}\geq
\frac{\gamma_n x_{\min}}{x_{\min}+(n-1)x_{\max}}.
\end{equation}
Since $D_{\max}-\lambda_1\leq\sigma_n$, we find that
\[
\frac{x_{\max}}{x_{\min}}\geq\frac{\gamma_n/\sigma_n-1}{n-1}.
\]
On the other hand, by \eqref{eq:up-ratio} we have
\[
\frac{x_{\max}}{x_{\min}}\leq \frac{1}{1-\sigma_n}.
\]
In view of \eqref{eq:equality-sigma-gamma} we see
\[
\left(\frac{\gamma_n}{\sigma_n}-1\right)(1-\sigma_n)=\frac{\gamma_n}{\sigma_n}+\sigma_n-\gamma_n-1=n-1,
\]
which yields that
\[
\frac{\gamma_n/\sigma_n-1}{n-1}=\frac{1}{1-\sigma_n}.
\]
Hence, all inequalities above must be equalities, we immediately obtain $x_{\max}/x_{\min}=(1-\sigma_n)^{-1}$.
Noting that $D_{\max}-\lambda_1\leq\sigma_n$, together with \eqref{eq:D-lambda-lower-bound}, we get 
$D_{\max}-\lambda_1=\sigma_n$, completing the proof of this claim.

\begin{claim}\label{claim:number-Dmax}
Let $S=\{w\in V(G^*): D_w=D_{\max}\}$. Then $|S|=n-1$.
\end{claim}

\noindent {\it Proof of \autoref{claim:number-Dmax}.}
If $n$ is odd, the result follows from \autoref{claim:nDmax-2W} immediately. If $n$ is even, then 
$|S|\in\{n-2,n-1\}$ by \autoref{claim:nDmax-2W}. In view of \autoref{claim:ratio}, there are $(n-1)$ 
vertices attaining $x_{\max}$. Hence, if $|S|=n-2$, there exists a vertex $w$ such that $x_w=x_{\max}$, 
while $D_w=D_{\max}-1=D_{\min}$, which is a contradiction with \autoref{claim:Du-Dmin-Dv-Dmax}.
Hence, $|S|=n-1$, as required.

\begin{claim}\label{claim:D-max=n}
$D_{\max}=n+\gamma_n-1$ and $D_{\min}=n-1$.
\end{claim}

\noindent {\it Proof of \autoref{claim:D-max=n}.}
According to \autoref{claim:ratio}, there are $(n-1)$ vertices attaining $x_{\max}$, and the remaining vertex 
$v$ attaining $x_{\min}$. By \autoref{claim:Du-Dmin-Dv-Dmax}, $D_v<D_{\max}$. Together with \autoref{claim:number-Dmax} 
we have $D_v=D_{\max}-\gamma_n=D_{\min}$. It follows that
\[
\lambda_1 x_v =\sum_{w\in V(G^*)} d(v,w) x_w =x_{\max} \sum_{w\in V(G^*)} d(v,w)=(D_{\max}-\gamma_n) x_{\max}.
\]
Recall that $D_{\max}-\lambda_1=\sigma_n$. Therefore we obtain
\[
D_{\max}-\sigma_n=\lambda_1=(D_{\max}-\gamma_n)\cdot\frac{x_{\max}}{x_{\min}}=\frac{D_{\max}-\gamma_n}{1-\sigma_n},
\]
which implies that
\[
D_{\max}=\sigma_n+\frac{\gamma_n}{\sigma_n}-1=n+\gamma_n-1.
\]
Hence, $D_{\min}=D_{\max}-\gamma_n=n-1$, completing the proof of this claim. \vspace{3mm}

Now, we continue our proof. Since $D_{\min}=n-1$, we get $d(v)=n-1$, which yields that the diameter of $G^*$ 
is two. As a consequence, for each vertex $w\in V(G^*)$ we have
\[
D_w=\sum_{z\in V(G^*)} d(w,z)=d(w)+2(n-1-d(w))=2(n-1)-d(w).
\]
Therefore, for each vertex $w\in V(G^*)\backslash\{v\}$ we get 
\[
d(w)=2(n-1)-D_w=n-1-\gamma_n=
\begin{cases}
        n-2, & \text{if $n$ is odd}, \\
        n-3, & \text{if $n$ is even}.
    \end{cases}
\]
Hence, $G^*\cong K_{1,2,\ldots,2}$ if $n$ is odd, and  $G^*$ is isomorphic to some $(n-4)$-DVDR graph if $n$ is even. 
\hfill $\Box$

\section{Concluding remarks}

In \autoref{thm:Main} we show that the extremal graph is an $(n-4)$-DVDR graph for even $n$. For any $(n-4)$-DVDR graph $G$ 
with distinguished vertex $v$, it is clear that the complement of $G-v$ is the disjoint union of some cycles. In addition, by 
simple algebra, we find that 
\[
\sigma(G)=\frac{n+2-\sqrt{n^2+4n-4}}{2}>\frac{1}{n/2+1}>\frac{1}{n+1}.
\] 

Let $H$ be an $r$-uniform hypergraph (i.e., a family of some $r$-element subsets of a set). The shadow graph of $H$ is the 
simple graph $\partial (H)$ on the same vertex set as $H$, where two vertices are adjacent if they are covered by at least 
one edge of $H$. Similarly, we can define the distance matrix for connected hypergraphs and non-transmission-regular hypergraphs. 
It is clear that $H$ and $\partial (H)$ have the same distance matrix for any connected uniform hypergraph $H$. Hence, we 
immediately obtain the following generalized results from \autoref{thm:Main}. 

\begin{theorem}
Let $H$ be a connected non-transmission-regular uniform hypergraph of order $n$. 
\begin{enumerate}
\item[$(1)$] If $n$ is odd, then
\[
D_{\max}(H)-\lambda_1(H)\geq\frac{n+1-\sqrt{(n-1)(n+3)}}{2}.
\] 
Equality holds if and only if $\partial (H)\cong K_{1,2,\ldots,2}$.

\item[$(2)$] If $n$ is even, then
\[
D_{\max}(H)-\lambda_1(H)\geq\frac{n+2-\sqrt{n^2+4n-4}}{2}.
\] 
Equality holds if and only if $\partial (H)$ is isomorphic to some $(n-4)$-DVDR graph.
\end{enumerate}
\end{theorem}   

It is worth noting that the extremal hypergraphs in the above theorem are not unique for uniformity at least three. Let $H_1$ 
and $H_2$ be $3$-uniform hypergraphs on the same vertex set $\{1,2,\ldots,7\}$, whose edge sets are $E_1$ and $E_2$, respectively, 
where
\[
\begin{aligned}
E_1 & =\{\{1,2,6\}, \{1,2,7\}, \{1,3,5\}, \{1,3,7\}, \{1,4,5\}, \{1,4,6\}, \{2,3,4\}, \{5,6,7\}\}, \\
E_2 & =\{\{1, 3, 5\}, \{1, 4, 6\}, \{1, 6, 7\}, \{2, 3, 6\}, \{2, 3, 7\}, \{2, 4, 5\}, \{4, 5, 7\}\}.
\end{aligned}
\]
It is clear that $H_1$ and $H_2$ have the same shadow graph $K_{1,2,2,2}$, while $H_1$ is not isomorphic to $H_2$.

\end{document}